\documentclass[12pt,a4paper]{article}

\usepackage{geometry}
\usepackage{amsmath,amssymb,amsthm}
\usepackage{tikz-cd}
\usepackage{multicol}
\usepackage{mathtools}
\usepackage{authblk}
\usepackage{accents}
\usepackage{bm}
\usepackage{graphicx}

\usepackage[bookmarks=false, hyperfootnotes=false, colorlinks=true, linkcolor=blue, citecolor=blue, urlcolor=blue]{hyperref}

\theoremstyle{plain}
\newtheorem{theorem}{Theorem}[section]
\newtheorem{lemma}[theorem]{Lemma}

\newtheorem{corollary}[theorem]{Corollary}

\theoremstyle{definition}

\newtheorem{problem}[theorem]{Problem}

\renewcommand{\ge}{\geqslant}

\DeclareMathOperator{\End}{End}
\DeclareMathOperator{\Aut}{Aut}
\DeclareMathOperator{\im}{im}
\DeclareMathOperator{\Cay}{Cay}

\newcommand{\diC}{\accentset{\raisebox{-0.4ex}[0pt][0pt]{\scalebox{0.72}[0.72]{$\bm{\rightharpoonup}$}}}{C}}

\title{Automorphism groups of endomorphism monoids of $0$-unit $3$-valent circulant digraphs}
\author{Chenhui Lv}
\affil{\small School of Mathematical Sciences, University of Science and Technology of China, Hefei 230026, People's Republic of China}
\date{\today}

\begin{document}
\maketitle

\begingroup
\renewcommand{\thefootnote}{}
\footnotetext{Corresponding author: Chenhui Lv. E-mail address: lch1994@mail.ustc.edu.cn.}
\addtocounter{footnote}{-1}
\endgroup

\begin{abstract}
Let $G=\Cay(\mathbb{Z}_n,S)$ be a $3$-valent circulant digraph with
connection set $S=\{0,s,t\}$, where $s,t\in\mathbb{Z}_n^*$ and
$t\neq\pm s$.
Determining the automorphism group of its endomorphism monoid
reduces to determining the subgroup
$U_S(\mathbb{Z}_n)\leq\mathbb{Z}_n^*$ that normalizes $\End(G)$.
In this paper, we first establish necessary and sufficient conditions
for a unit circulant digraph without $2$-cycles to admit an
endomorphism whose image induces a directed cycle.
Using this characterization, we explicitly determine $U_S(\mathbb{Z}_n)$ for the previously unresolved $0$-unit $3$-valent family.
\end{abstract}

\noindent\textbf{Keywords:} circulant digraph; endomorphism monoid; automorphism group; core.

\medskip
\noindent\textbf{2020 Mathematics Subject Classification:} 05C20, 05C25, 20M20.

\section{Introduction}\label{sec:introduction}

Endomorphism monoids of various mathematical objects encode the self-maps of their structures, and determining the automorphism groups of the endomorphism monoids is a fundamental problem in semigroup theory; see~\cite{AK2007,AK2009a,AK2009b,BBL2007,BL2011,MS2003,MSZ2006,Z2017,Z2026}.

A \emph{digraph} (or \emph{directed graph}) $G = (V, A)$ consists of a non-empty set of vertices $V$ and a set of arcs $A \subseteq V \times V$.
Write $x\to y$ for an arc $(x,y)\in A$, and denote the endomorphism monoid (see Section~\ref{sec:preliminaries} for the definition) of $G$ by $\End(G)$.
In this paper, we study the endomorphism monoids of digraphs.
A fundamental result in this direction is due to Hedrl\'{\i}n and Pultr~\cite{HP1965}, who proved that every monoid of cardinality less than the first inaccessible cardinal can be realized as the endomorphism semigroup of an undirected graph.

The digraph $G$ is called \emph{reflexive} if $(x, x) \in A$ for all $x \in V$, and \emph{anti-reflexive} if $(x, x) \notin A$ for all $x \in V$.
In 1972, Va\v{z}enin~\cite{V1972} described the automorphisms of $\End(G)$ when $G$ is a reflexive digraph containing a non-loop arc that lies on no directed cycle of length at least $2$.
Note that the reflexivity of $G$ ensures that $\End(G)$ is a separated transformation semigroup, allowing us to apply Schreier's result~\cite{Schreier1937} to study $\Aut(\End(G))$.

For a positive integer $n$, let $\mathbb{Z}_n^*$ denote the group of units of $\mathbb{Z}_n$. We identify each $k\in\mathbb{Z}_n^*$ with the multiplier $x\mapsto kx$, and hence identify $\mathbb{Z}_n^*$ with $\Aut(\mathbb{Z}_n)$. For a subset $S\subseteq\mathbb{Z}_n$, the \emph{circulant digraph} of order $n$ with \emph{connection set} $S$ is the Cayley digraph $G=\Cay(\mathbb{Z}_n,S)$ with vertex set $\mathbb{Z}_n$ and arc set $\{(i,j)\in\mathbb{Z}_n\times\mathbb{Z}_n\mid j-i\in S\}$.
Every vertex of $G$ has in-valency and out-valency $|S|$.
The digraph $G$ is reflexive precisely when $0\in S$.
We call $G$ a \emph{unit circulant digraph} if $S\subseteq\mathbb{Z}_n^*$, and a \emph{$0$-unit circulant digraph} if $0\in S$ and $S\setminus\{0\}\subseteq\mathbb{Z}_n^*$.
For $s\in S$, an arc $i\to j$ of $\Cay(\mathbb{Z}_n,S)$ is said to have \emph{color} $s$ if $j-i=s$.

Ara\'ujo, Dobson, and Konieczny~\cite[Section~4]{ADK2010} and Ara\'ujo, Bentz, Dobson, Konieczny, and Morris~\cite[Section~6]{ABDKM2018} determined the automorphism groups of the endomorphism monoids of several classes of reflexive circulant digraphs, including digraphs with $2$-cycles and certain reflexive $3$-valent circulants.
Ara\'ujo, Dobson, and Konieczny~\cite{ADK2010} also determined the automorphism groups of the endomorphism monoids of two other classes of reflexive digraphs, namely, non-cyclic digraphs and permutational digraphs.
For $S \subseteq \mathbb{Z}_n$, define the group
\[
U_S(\mathbb{Z}_n) = \{k \in \mathbb{Z}_n^* \mid k \End(\Cay(\mathbb{Z}_n,S)) k^{-1} = \End(\Cay(\mathbb{Z}_n,S))\}.
\]
The following theorem was proved in~\cite{ABDKM2018}.

\begin{theorem}[{\cite[Theorem~6.4]{ABDKM2018}}]\label{thm:aut-of-end}
Let $G=\Cay(\mathbb{Z}_n,S)$ be a reflexive circulant digraph of order $n\geq2$, and let $p$ be the smallest prime divisor of $n$. If $n$ is square-free, or $|S|\leq p+1$, or $G$ is a $0$-unit circulant digraph, then
\begin{equation}\label{eq:aut-us}
\Aut(\End(G))=\Aut(G)\,U_S(\mathbb{Z}_n).
\end{equation}
\end{theorem}

As a special case of Theorem~\ref{thm:aut-of-end}, for all reflexive $3$-valent circulants $G=\Cay(\mathbb{Z}_n,S)$, the isomorphism~\eqref{eq:aut-us} holds. Thus, to determine $\Aut(\End(G))$ for such circulants $G$, it suffices to determine $U_S(\mathbb{Z}_n)$.
In Section~6.2, Ara\'ujo et al.~\cite{ABDKM2018} considered the following $0$-unit subfamily, which forms part of their unresolved Problem~3:

\begin{problem}\label{prob:problem-3}
Let $n \geq 3$ be an integer, and let $S=\{0,s,t\}$ with
$s,t\in\mathbb{Z}_n^*$ and $t\neq \pm s$.
Determine $U_S(\mathbb{Z}_n)$.
\end{problem}

The following theorem determines $U_S(\mathbb{Z}_n)$, and hence $\Aut(\End(G))$ by~\eqref{eq:aut-us}, for this family.

\begin{theorem}\label{thm:grand-classification}
Let $n \geq 3$ be an integer, let $S=\{0,s,t\}$ with $s,t\in\mathbb{Z}_n^*$ and $t\neq\pm s$, and set $d=\gcd(n,s-t)$.
The following statements hold:

\begin{enumerate}
\item[\textup{(a)}] If $2d<n$ and the equation $xs+(d-x)t\equiv0\pmod n$ has an integer solution $1\leq x\leq d-1$, then
\[
U_S(\mathbb{Z}_n)=
\begin{cases}
\{\pm1,\pm ts^{-1}\}, & t^2\equiv s^2\pmod n,\\[1mm]
\{\pm1\}, & t^2\not\equiv s^2\pmod n.
\end{cases}
\]
\item[\textup{(b)}] Otherwise, $U_S(\mathbb{Z}_n)=\mathbb{Z}_n^*$.
\end{enumerate}
\end{theorem}

The proof of Theorem~\ref{thm:grand-classification} requires determining when the corresponding anti-reflexive digraph $G'=\Cay(\mathbb{Z}_n,S\setminus\{0\})$ admits an endomorphism whose image induces a directed cycle.
With this motivation, we establish in Section~\ref{sec:cycle-endomorphisms} an exact characterization (Theorem~\ref{thm:cycle-image-characterization}) of when a unit circulant digraph without $2$-cycles admits an endomorphism onto a directed cycle, which is also of independent interest.
Based on this, in Section~\ref{sec:proof-of-main}, we explicitly
determine the group $U_S(\mathbb{Z}_n)$, thereby proving
Theorem~\ref{thm:grand-classification} and completing the
determination of $U_S(\mathbb{Z}_n)$ for the unresolved $0$-unit
family described above.

\section{Preliminaries}\label{sec:preliminaries}

Let $G = (V, A)$ and $H = (V', A')$ be digraphs.
For a non-empty subset $W \subseteq V$, let $G[W]$ denote the \emph{sub-digraph induced by $W$}. A \emph{homomorphism} $f: G \to H$ is a map from $V$ to $V'$ that sends arcs to arcs (non-arcs may be sent to arcs as well).
An \emph{endomorphism} of $G$ is a homomorphism from $G$ to itself.
The set of all endomorphisms of $G$, denoted by $\End(G)$, forms a monoid under composition. An endomorphism $f \in \End(G)$ is called \emph{non-trivial} if it is neither an automorphism of $G$ nor a constant map.

Let $H$ be a sub-digraph of a digraph $G$. A \emph{retraction} from $G$ to $H$ is a homomorphism $r : G \to H$ such that $r(x) = x$ for all $x \in V(H)$. If there is a retraction from $G$ to $H$, we say that $G$ \emph{retracts} to $H$ and that $H$ is a \emph{retract} of $G$.
A \emph{core} is a finite digraph that does not retract to any proper sub-digraph.

\begin{lemma}[{\cite[Proposition~1.31]{HN2004}}]\label{lem:core-proper-subdigraph}
A finite digraph $G$ is a core if and only if there is no homomorphism from $G$ to a proper sub-digraph.
\end{lemma}

Every finite digraph $G$ has, up to isomorphism, a unique retract that is a core. If $H$ is such a retract, then $H$ is called the \emph{core} of $G$. See~\cite[Section~1.6]{HN2004} for further details.
The corresponding result for graphs was proved by Welzl~\cite{W1984} and independently by MacGillivray; see also Hahn and Tardif~\cite[Theorem~3.7]{HT1997}. The same proof applies verbatim to finite digraphs.

\begin{theorem}\label{thm:core-vertex-transitive}
The core of every finite vertex-transitive digraph is vertex-transitive.
\end{theorem}

For a positive integer $\ell$, we write $\diC_\ell$ for a directed cycle of length $\ell$; in particular, $\diC_1$ is a loop, and $\diC_2$ is the directed $2$-cycle.
A digraph $G$ is called \emph{strongly connected} if, for every pair of distinct vertices $x, y \in V(G)$, there exists a directed walk from $x$ to $y$.

For a subset $S\subseteq\mathbb{Z}_n$ and an element $k\in\mathbb{Z}_n$, write
\[
-S=\{-s\mid s\in S\}\ \text{ and }\ kS=\{ks\mid s\in S\}.
\]
With this notation, conjugation by $k\in\mathbb{Z}_n^*$ induces an isomorphism
\[
\End(\Cay(\mathbb{Z}_n,S))
\to
\End(\Cay(\mathbb{Z}_n,kS)),
\ \
f\mapsto kfk^{-1}.
\]
Thus, the units $k \in \mathbb{Z}_n^*$ such that
$k\End(\Cay(\mathbb{Z}_n, S))k^{-1} = \End(\Cay(\mathbb{Z}_n, S))$ are precisely those satisfying
$\End(\Cay(\mathbb{Z}_n, S)) = \End(\Cay(\mathbb{Z}_n, kS))$. In other words,
\begin{equation}\label{eq:key-expression}
U_S(\mathbb{Z}_n)
= \{k \in \mathbb{Z}_n^* \mid \End(\Cay(\mathbb{Z}_n, S)) = \End(\Cay(\mathbb{Z}_n, kS))\}.
\end{equation}

\section{Non-trivial endomorphisms to directed cycles}\label{sec:cycle-endomorphisms}

In this section, we establish necessary and sufficient conditions characterizing when a $0$-unit $3$-valent circulant digraph admits a non-trivial endomorphism.

Consider first an anti-reflexive circulant digraph $G=\Cay(\mathbb{Z}_n,\{s,t\})$, where $s,t\in\mathbb{Z}_n^*$ and $t\neq\pm s$.
Suppose that $G$ admits a non-trivial endomorphism. Then Lemma~\ref{lem:core-proper-subdigraph} implies that $G$ is not a core, and hence the core $C$ of $G$ is a proper sub-digraph of $G$.
Since $s \in \mathbb{Z}_n^*$, the arcs of color $s$ form a directed Hamiltonian cycle of $G$. Therefore, $G$ is strongly connected, and so is its homomorphic image $C$. By Theorem~\ref{thm:core-vertex-transitive}, the digraph $C$ is vertex-transitive. Since $G$ has arcs and $C$ is anti-reflexive, $C$ cannot have only one vertex. Every vertex of $C$ has the same out-valency, which is either $1$ or $2$. If this out-valency were $2$, then no arc of $G$ would go from $V(C)$ to $V(G) \setminus V(C)$, and the strong connectivity of $G$ would force $C=G$, a contradiction. Thus $C$ has out-valency $1$ and is therefore isomorphic to $\diC_d$ for some $d$. Since $G$ contains no $2$-cycles, we have $d\geq3$.

The above consideration shows that homomorphisms onto directed cycles play a central role in determining whether such circulant digraphs admit non-trivial endomorphisms. We therefore address the following more general question: under what conditions does a unit circulant digraph without $2$-cycles admit an endomorphism whose image induces a directed cycle? The next theorem provides a complete characterization.

\begin{theorem}\label{thm:cycle-image-characterization}
Let $G=\Cay(\mathbb{Z}_n,S)$, where $\varnothing\neq S\subseteq\mathbb{Z}_n^*$ and $S\cap(-S)=\varnothing$, and let $d\geq3$ be an integer. Then there exists an endomorphism $f\in\End(G)$ such that the induced sub-digraph $G[\im(f)]$ is the directed cycle $\diC_d$ if and only if the following conditions hold:
\begin{enumerate}
\item[\textup{(a)}] $d\mid n$, and there exists an integer $c$ such that $s\equiv c\pmod d$ for every $s\in S$.
\item[\textup{(b)}] There exist $s_1,\dots,s_d\in S$, not necessarily distinct, such that $\sum_{i=1}^d s_i\equiv0\pmod n$.
\end{enumerate}
Moreover, under these conditions, the core of $G$ is $\diC_d$, and the length of every directed cycle in $G$ (whether induced or not) is divisible by $d$.
\end{theorem}

\begin{proof}
First suppose that conditions~\textup{(a)} and~\textup{(b)} hold. Let
\[
v_0=0 \ \text{ and }\ v_k=\sum_{i=1}^k s_i\in \mathbb{Z}_n
\ \text{ for } k \in \{1,2,\ldots, d\}.
\]
Then $v_d=v_0$, and $v_k\to v_{k+1}$ is an arc of $G$ for every $k\in\{0,1,\ldots,d-1\}$.
Since every $s\in S$ is a unit modulo $n$, we have $\gcd(s,n)=1$. Together with $s\equiv c\pmod d$ and $d\mid n$, this implies $\gcd(c,d)=1$.
Thus $c$ is invertible modulo $d$. Let $c^{-1}$ denotes the inverse of $c$ in $\mathbb{Z}_d^*$. Define
\[
f:\mathbb{Z}_n\to\mathbb{Z}_n,
\ \
f(x)=v_{c^{-1}x \bmod d}.
\]
If $x\to x+s$ is an arc of $G$, then $c^{-1}(x+s)\equiv c^{-1}x+c^{-1}s\equiv c^{-1}x+1\pmod d$, and so $f(x)\to f(x+s)$ is one of the arcs $v_k\to v_{k+1}$. Therefore, $f\in\End(G)$.
We now show that the induced sub-digraph $G[\im(f)]$ is precisely $\diC_d$. For each $k\in\{0,1,\ldots,d-1\}$, condition~\textup{(a)} gives
\[
c^{-1}v_k=c^{-1}\sum_{i=1}^k s_i\equiv c^{-1}kc\equiv k\pmod d.
\]
Hence $f(v_k)=v_k$. In particular, $v_0,v_1,\dots,v_{d-1}$ are pairwise distinct. Indeed, if $v_i=v_j$ with $0\leq i,j\leq d-1$, then the preceding congruence gives $i\equiv c^{-1}v_i=c^{-1}v_j\equiv j\pmod d$, which forces $i=j$. Moreover, if $v_i\to v_j$ is an arc of $G$, then $v_j-v_i\equiv c\pmod d$, so $j-i\equiv c^{-1}(v_j-v_i)\equiv1\pmod d$.
Thus $G[\im(f)]=\diC_d$.

Since $f(v_k)=v_k$ for each $k\in\{0,1,\ldots,d-1\}$, the map $f$ is a retraction from $G$ onto $G[\im(f)]$. Note that every endomorphism of a directed cycle is an automorphism, and so $\diC_d$ is a core. Therefore, the core of $G$ is $\diC_d$.
Let $u_0\to \ldots\to u_{h-1}\to u_h=u_0$ be a directed cycle of length $h$ in $G$. For each $i\in\{0,1,\ldots,h-1\}$, condition~\textup{(a)} gives $u_{i+1}-u_i\equiv c\pmod d$. Summing around the cycle yields
\[
0\equiv u_h-u_0\equiv\sum_{i=0}^{h-1}(u_{i+1}-u_i)\equiv hc\pmod d.
\]
Since $\gcd(c,d)=1$, it follows that $h\equiv0\pmod d$, whence the length of every directed cycle in $G$ is divisible by $d$.

Conversely, suppose that $f\in\End(G)$ and that $G[\im(f)]$ is $\diC_d$. Label the vertices of $\diC_d$ as $v_0,\dots,v_{d-1}$ so that its arcs are $v_i\to v_{i+1}$, with indices taken modulo $d$. For each $x\in\mathbb{Z}_n$, let $k_x\in\mathbb{Z}_d$ be the unique index such that $f(x)=v_{k_x}$.
Since $f$ is an endomorphism, the arc $x\to x+s$ is mapped to the arc $v_{k_x}\to v_{k_{x+s}}$ in $G[\im(f)]$ for every $x\in\mathbb{Z}_n$ and $s\in S$. As $G[\im(f)]$ is a directed cycle, it follows that
\begin{equation}\label{eq:index-shift}
k_{x+s}-k_x\equiv1\pmod d.
\end{equation}
Fix $s_0\in S$. Iterating~\eqref{eq:index-shift} $n$ times along arcs of color $s_0$ yields $k_{x+ns_0}-k_x\equiv n\pmod d$. Since $ns_0=0$ in $\mathbb{Z}_n$, this implies that $d\mid n$.

Let $b\in\mathbb{Z}_n^*$ satisfy $bs_0\equiv1\pmod n$. For each $s\in S$, choose $q\in\{0,1,\ldots,n-1\}$ such that $q\equiv bs\pmod n$. Then $qs_0\equiv s\pmod n$. Iterating~\eqref{eq:index-shift} $q$ times along arcs of color $s_0$ and comparing the result with one application along an arc of color $s$ gives
\[
q\equiv k_{x+qs_0}-k_x=k_{x+s}-k_x\equiv1\pmod d.
\]
Since $d\mid n$, reducing $qs_0\equiv s\pmod n$ modulo $d$ and using $q\equiv1\pmod d$ gives $s\equiv s_0\pmod d$ for every $s\in S$. Thus condition~\textup{(a)} holds.
Finally, for each arc $v_{i-1}\to v_i$ of $G[\im(f)]$, choose $s_i\in S$ such that $v_i-v_{i-1}=s_i$ in $\mathbb{Z}_n$. Summing these equalities around the cycle yields condition~\textup{(b)}.
\end{proof}

We now return to $0$-unit $3$-valent circulant digraphs, for which we need the following lemma.

\begin{lemma}[{\cite[Lemma~6.7]{ABDKM2018}}]\label{lem:constant-endomorphism}
Let $n\ge3$ be an integer, and let $G = \Cay(\mathbb{Z}_n,S)$, where $S=\{0,s,t\}$ with $s,t\in \mathbb{Z}_n^*$ and $t \neq \pm s$. Then $f \in \End(G)$ is constant if and only if there exist $y\in \mathbb{Z}_n$ and $u\in \{s,t\}$ such that $f(y+u)=f(y)$.
\end{lemma}

We are now ready to conclude the section with the following corollary of Theorem~\ref{thm:cycle-image-characterization}.

\begin{corollary}\label{cor:main-conditions}
Let $n\ge3$ be an integer, and let $G=\Cay(\mathbb{Z}_n,S)$, where $S=\{0,s,t\}$ with $s,t\in\mathbb{Z}_n^*$ and $t\neq\pm s$.
Then $G$ has a non-trivial endomorphism if and only if there exists a divisor $d\ge3$ of $n$ satisfying the following conditions:
\begin{enumerate}
\item[\textup{(a)}] $s \equiv t \pmod d$.
\item[\textup{(b)}] The equation $xs+(d-x)t \equiv 0 \pmod n$ has a solution in $\{1,2,\ldots,d-1\}$.
\end{enumerate}
Moreover, if these conditions hold, then $d=\gcd(n,s-t)$. The core of $\Cay(\mathbb{Z}_n,\{s,t\})$ is $\diC_d$, and the length of every directed cycle (whether induced or not) in $\Cay(\mathbb{Z}_n,\{s,t\})$ is divisible by $d$.
\end{corollary}

\begin{proof}
Set $S'=S\setminus\{0\}=\{s,t\}$ and $G'=\Cay(\mathbb{Z}_n,S')$. Then $S'\subseteq\mathbb{Z}_n^*$ and $S'\cap(-S')=\varnothing$. $G'$ is obtained from $G$ by removing the loop at every vertex, and hence $\End(G')\subseteq\End(G)$ and $\Aut(G')=\Aut(G)$. Conversely, if $f\in\End(G)$ is non-constant, then Lemma~\ref{lem:constant-endomorphism} implies that $f\in\End(G')$.

Suppose first that $G$ has a non-trivial endomorphism. The preceding observation gives a non-trivial endomorphism of $G'$, so Lemma~\ref{lem:core-proper-subdigraph} implies that $G'$ is not a core. Let $C$ be the core of $G'$. Then $C$ is a proper retract of $G'$.
Since $s\in\mathbb{Z}_n^*$, the arcs of color $s$ form a directed Hamiltonian cycle of $G'$, and hence $G'$ is strongly connected. As $C$ is a retract of $G'$, it is a homomorphic image of $G'$ and is therefore strongly connected. By Theorem~\ref{thm:core-vertex-transitive}, $C$ is vertex-transitive. Moreover, $C$ is anti-reflexive and has at least two vertices.
Every vertex of $C$ has the same out-valency, which is either $1$ or $2$, since $C$ is a sub-digraph of the $2$-valent digraph $G'$. If this out-valency were $2$, then no arc of $G'$ would go from $V(C)$ to $V(G')\setminus V(C)$.
The strong connectivity of $G'$ would then force $C=G'$, a contradiction. Since $C$ is proper, its out-valency is therefore $1$. As $C$ is finite and strongly connected, it follows that $C\cong\diC_d$ for some $d$. Since $G'$ is anti-reflexive and contains no $2$-cycles, we have $3\leq d<n$.

Because $C$ is a retract of $G'$, there exists a retraction $r:G'\to C$. Furthermore, $C$ is induced on its vertex set: indeed, if $x,y\in V(C)$ and $x\to y$ is an arc of $G'$, then $r(x)=x$ and $r(y)=y$, so the homomorphism property of $r$ implies that $x\to y$ is an arc of $C$. Hence $G'[\im(r)]=C=\diC_d$.
Theorem~\ref{thm:cycle-image-characterization} now applies. Its first condition gives $d\mid n$ and $s\equiv t\pmod d$. Its second condition provides an $x\in\{0,1,\ldots,d\}$ such that $xs+(d-x)t\equiv0\pmod n$.
As $s,t\in\mathbb{Z}_n^*$, we have $x\neq0$ and $x\neq d$, and thus $x\in\{1,2,\ldots,d-1\}$, which proves necessity.

Conversely, suppose that a divisor $d\geq3$ of $n$ satisfies conditions~\textup{(a)} and~\textup{(b)}. Condition~\textup{(a)} gives $s\equiv t\pmod d$; since $s\neq t$ in $\mathbb{Z}_n$, we have $d<n$.
Condition~\textup{(b)} provides $d$ elements of $\{s,t\}$ whose sum is congruent to $0$ modulo $n$.
Theorem~\ref{thm:cycle-image-characterization} therefore yields an endomorphism $r\in\End(G')$ such that $G'[\im(r)]=\diC_d$. The map $r$ is also an endomorphism of $G$. Since $3\leq d<n$, it is non-trivial.
Furthermore, Theorem~\ref{thm:cycle-image-characterization} implies that the core of $G'$ is $\diC_d$, and that the length of every directed cycle (whether induced or not) in $G'$ is divisible by $d$. To determine $d$, write $n=dh$ and $s-t=dq$. Condition~\textup{(b)} gives $xq+t\equiv0\pmod h$, so $\gcd(q,h)\mid t$. Since $h\mid n$ and $\gcd(t,n)=1$, we obtain $\gcd(q,h)=1$. Thus $\gcd(n,s-t)=d\gcd(h,q)=d$, proving uniqueness.
\end{proof}

\section{Proof of Theorem~\ref{thm:grand-classification}}\label{sec:proof-of-main}

In this section, we first establish a criterion for $\End(\Cay(\mathbb{Z}_n,S))=\End(\Cay(\mathbb{Z}_n,kS))$, and then apply it to prove Theorem~\ref{thm:grand-classification}.

\begin{lemma}\label{lem:end-monoid-criterion}
Let $S\subseteq \mathbb{Z}_n$ and $k\in \mathbb{Z}_n^*$. Then $\End(\Cay(\mathbb{Z}_n,S))=\End(\Cay(\mathbb{Z}_n,kS))$ if and only if $g(y+kS)\subseteq g(y)+kS$ for every $g\in \End(\Cay(\mathbb{Z}_n,S))$ and every $y\in \mathbb{Z}_n$.
\end{lemma}
\begin{proof}
The necessity is immediate. We prove sufficiency. Let $g\in \End(\Cay(\mathbb{Z}_n,S))$, and let $(y,z)$ be an arbitrary arc of $\Cay(\mathbb{Z}_n,kS)$. Then $z\in y+kS$.
By hypothesis, $g(z)\in g(y)+kS$, and therefore $(g(y),g(z))$ is an arc of $\Cay(\mathbb{Z}_n,kS)$. Hence $g\in \End(\Cay(\mathbb{Z}_n,kS))$, and so $\End(\Cay(\mathbb{Z}_n,S))\subseteq \End(\Cay(\mathbb{Z}_n,kS))$.
Since $k\in \mathbb{Z}_n^*$, the map $y\mapsto ky$ induces an isomorphism $\Cay(\mathbb{Z}_n,S)\cong \Cay(\mathbb{Z}_n,kS)$. Therefore, their endomorphism monoids have the same cardinality, and the inclusion above forces $\End(\Cay(\mathbb{Z}_n,S))=\End(\Cay(\mathbb{Z}_n,kS))$.
\end{proof}

We are now ready to prove Theorem~\ref{thm:grand-classification}.

\begin{proof}[\textbf{Proof of Theorem~\ref{thm:grand-classification}}]
Set $G=\Cay(\mathbb{Z}_n,S)$ and $d=\gcd(n,s-t)$. Since $s\neq t$, $d$ is a proper divisor of $n$, so $2d\leq n$. Moreover, $2(s-t)\equiv0\pmod n$ if and only if $2d=n$.
Throughout this proof, whenever an element of $\mathbb{Z}_n$ is used as an integer, we take its representative in $\{0,1,\ldots,n-1\}$. Let $k\in\mathbb{Z}_n^*$. By equation~\eqref{eq:key-expression}, $k\in U_S(\mathbb{Z}_n)$ if and only if $\End(\Cay(\mathbb{Z}_n,S))=\End(\Cay(\mathbb{Z}_n,kS))$.
Let $S'=\{s,t\}$ and $kS'=\{ks,kt\}$.
By Lemma~\ref{lem:constant-endomorphism}, every non-constant endomorphism of $\Cay(\mathbb{Z}_n,S)$ (resp.~$\Cay(\mathbb{Z}_n,kS)$) is also an endomorphism of the corresponding anti-reflexive sub-digraph $\Cay(\mathbb{Z}_n,S')$  (resp.~$\Cay(\mathbb{Z}_n,kS')$).
Therefore, we have $\End(\Cay(\mathbb{Z}_n,S))=\End(\Cay(\mathbb{Z}_n,kS))$ if and only if $\End(\Cay(\mathbb{Z}_n,S'))= \End(\Cay(\mathbb{Z}_n,kS'))$.

We first record that every $k\in\mathbb{Z}_n^*$ satisfying $kS=S$ or $kS=-S$ belongs to $U_S(\mathbb{Z}_n)$.
Indeed, if $kS=S$, then $\End(\Cay(\mathbb{Z}_n,kS))=\End(\Cay(\mathbb{Z}_n,S))$. If $kS=-S$, then $\Cay(\mathbb{Z}_n,kS)=\Cay(\mathbb{Z}_n,-S)=G^{\mathrm{op}}$, where $G^{\mathrm{op}}$ is obtained from $G$ by reversing every arc. A map is an endomorphism of a digraph if and only if it is an endomorphism of its opposite digraph, so $\End(G^{\mathrm{op}})=\End(G)$. Thus in either case $k\in U_S(\mathbb{Z}_n)$.

Suppose first that the equation $xs+(d-x)t\equiv0\pmod n$ has no integer solution $1\leq x\leq d-1$. Then Corollary~\ref{cor:main-conditions} implies that $\End(G)$ consists solely of automorphisms and constant maps. For $k\in\mathbb{Z}_n^*$, we have $k \End(G)k^{-1} = \End(G)$ if and only if $k \Aut(G)k^{-1} = \Aut(G)$.
By~\cite[Proposition~6.3]{ABDKM2018}, $k \Aut(G)k^{-1}=\Aut(G)$ for every $k\in\mathbb{Z}_n^*$, and therefore $U_S(\mathbb{Z}_n)=\mathbb{Z}_n^*$. This proves Item~\textup{(b)} in this case.

For the remainder of the proof, suppose that the equation has an integer solution $1\leq x\leq d-1$. This excludes $d=1$. If $d=2$, then $x=1$ and $s+t\equiv0\pmod n$, contrary to $t\neq-s$. Hence $3\leq d<n$.

We next consider the case where $2d=n$. Then $2(s-t)\equiv0\pmod n$ and $s-t\not\equiv0\pmod n$, so $s-t\equiv d\pmod n$.
For $k\in\mathbb{Z}_n^*$, $k$ is odd, so $n=2d$ gives $kd\equiv d\pmod n$. Thus $kS'=\{k(t+d),kt\}=\{kt,kt+d\}$.
In $\Cay(\mathbb{Z}_n,S')$, the endpoint of a walk of length $k$ starting at a vertex $y$ with $m$ arcs of color $s$ and $k-m$ arcs of color $t$ is $y+ms+(k-m)t\equiv y+kt+m(s-t)\equiv y+kt+md\pmod n$, where $0\leq m\leq k$.
Since $n=2d$, the term $md\pmod n$ is $0$ if $m$ is even and $d$ if $m$ is odd. Hence the set of endpoints of walks of length $k$ starting at $y$ is precisely $\{y+kt,y+kt+d\}=y+kS'$.
Let $g\in \End(\Cay(\mathbb{Z}_n,S'))$.
Let $z\in y+kS'$. By the preceding observation, there is a walk of length $k$ from $y$ to $z$. Applying $g$ to this walk gives a walk of length $k$ from $g(y)$ to $g(z)$. Hence $g(z)\in g(y)+kS'$. Thus $g(y+kS')\subseteq g(y)+kS'$. By Lemma~\ref{lem:end-monoid-criterion}, we obtain $\End(\Cay(\mathbb{Z}_n,S'))=\End(\Cay(\mathbb{Z}_n,kS'))$. Therefore, $U_S(\mathbb{Z}_n)=\mathbb{Z}_n^*$, proving Item~\textup{(b)} in this case.

For the remainder of the proof, suppose that $2d<n$, or equivalently, $2(s-t)\not\equiv0\pmod n$.
Let $k\in\mathbb{Z}_n^*$, and let $c\in\{1,2,\ldots,d-1\}$ be the unique integer satisfying $c\equiv k\pmod d$. Such a $c$ exists because $d\mid n$ and $\gcd(k,n)=1$ imply $\gcd(k,d)=1$.

We first deduce a necessary condition for $k\in U_S(\mathbb{Z}_n)$. Assume that $\End(\Cay(\mathbb{Z}_n,S')) = \End(\Cay(\mathbb{Z}_n,kS'))$. For any solution $x\in\{1,2,\ldots,d-1\}$ of $xs+(d-x)t\equiv0\pmod n$, choose any ordering of $x$ arcs of color $s$ and $d-x$ arcs of color $t$. This determines a closed walk $W$ of length $d$ in $\Cay(\mathbb{Z}_n,S')$, starting and ending at $0$. Write its vertex sequence as $w_0,w_1,\ldots,w_d, \  w_0=w_d=0$.
Since $s\in\mathbb{Z}_n^*$ and $d\mid n$, the image of $s$ is a unit modulo $d$; in the definition below, $s^{-1}$ denotes its inverse in $\mathbb{Z}_d^*$. Define
\[
f_W:\mathbb{Z}_n\to\mathbb{Z}_n,
\ f_W(y)=w_{ys^{-1}\bmod d}.
\]
For $y\in\mathbb{Z}_n$, let $i\in\{0,1,\ldots,d-1\}$ satisfy $i\equiv ys^{-1}\pmod d$. Every $u\in S'$ satisfies $u\equiv s\pmod d$, so $f_W(y+u)-f_W(y)=w_{i+1}-w_i\in S'$, where $w_d=w_0$. Thus $f_W\in\End(\Cay(\mathbb{Z}_n,S')) $ and $ f_W(0)=0$. By the assumed equality of the endomorphism monoids, $f_W\in\End(\Cay(\mathbb{Z}_n,kS'))$. Since $f_W(0)=0$, $f_W(kS')\subseteq kS'$. Moreover, $ts^{-1}\equiv1\pmod d$ and $k\equiv c\pmod d$, so $f_W(ks)=f_W(kt)=w_c$. Hence every such closed walk satisfies
\begin{equation}\label{eq:wc-in-ks}
w_c\in kS'.
\end{equation}

We now consider the possible values of $c$, where $1\leq c\leq d-1$.
If $c=1$, choose two such closed walks whose first arcs have colors $s$ and $t$, respectively. This is possible because every solution satisfies $1\leq x\leq d-1$, so both colors occur. By \eqref{eq:wc-in-ks}, $S'\subseteq kS'$. Since both sets have two elements, $kS'=S'$, and hence $kS=S$.
Similarly, if $c=d-1$, choose closed walks whose last arcs have colors $s$ and $t$, respectively. Their $(d-1)$st vertices are $-s$ and $-t$, so \eqref{eq:wc-in-ks} gives $-S'\subseteq kS'$. Again both sets have two elements, so $kS'=-S'$ and hence $kS=-S$.
Thus, if $c=1$ or $c=d-1$, then $kS=S$ or $kS=-S$, respectively. It remains to rule out $2\leq c\leq d-2$.

Suppose that the equation $xs+(d-x)t\equiv0\pmod n$ has a solution $x\in\{2,3,\ldots,d-2\}$.
Fix such a solution $x$. Among the first $c$ steps of a corresponding closed walk, let $m$ be the number of arcs of color $s$. Then $w_c(m) \equiv ms+(c-m)t \equiv ct+m(s-t) \pmod n$. Every integer $m$ in the range $m_{\min}=\max(0,c-d+x) \leq m\leq m_{\max}=\min(c,x)$ is realized by some ordering of the $x$ arcs of color $s$ and $d-x$ arcs of color $t$.
Since $2\leq c\leq d-2$ and $2\leq x\leq d-2$, we have $m_{\max}-m_{\min}=\min(c,x,d-c,d-x)\geq2$. Choose three consecutive values $m_0,m_0+1,m_0+2$. The corresponding closed walks have
\begin{align*}
y_0 &\equiv ct+m_0(s-t) \pmod n,\\
y_1 &\equiv y_0+(s-t) \pmod n,\\
y_2 &\equiv y_0+2(s-t) \pmod n
\end{align*}
as their respective vertices after $c$ steps. By \eqref{eq:wc-in-ks}, all three vertices belong to $kS'$. However, because $s\neq t$ and $2(s-t)\not\equiv0\pmod n$, they are pairwise distinct. This is impossible since $|kS'|=2$.
Consequently, no value of $c$ with $2\leq c\leq d-2$ is possible in this case.

Suppose that the equation $xs+(d-x)t\equiv0\pmod n$ has no solution $x\in\{2,3,\ldots,d-2\}$.
Since it has a solution in $\{1,2,\ldots,d-1\}$, interchanging $s$ and $t$ if necessary, we may assume that $x=1$ is a solution. Set $h=n/d$. Since $d\mid(s-t)$, $h(s-t)\equiv0\pmod n$, so $x=1+h$ is also a solution. The absence of solutions in $\{2,3,\ldots,d-2\}$ gives $1+h\geq d-1$, and hence $h\geq d-2$.
The relation $s+(d-1)t\equiv0\pmod n$ gives $s\equiv(1-d)t\pmod n$. For $2\leq c\leq d-2$, consider the closed walks corresponding to the solution $x=1$. Choosing the unique $s$-arc after the first $c$ steps gives $w_c=ct$, while choosing it among the first $c$ steps gives $w_c=s+(c-1)t$. Therefore, by \eqref{eq:wc-in-ks}, $\{ct,s+(c-1)t\}\subseteq kS'$. The two elements on the left are distinct, and $|kS'|=2$, so $\{ct,s+(c-1)t\}=kS'$. Multiplying by $t^{-1}$ gives $\{c,c-d\}=\{k(1-d),k\}$ in $\mathbb{Z}_n$.
There are two possible matchings. If $k\equiv c\pmod n $ and $ k(1-d)\equiv c-d\pmod n$, then $d(c-1)\equiv0\pmod n$, so $h\mid(c-1)$. If instead $k\equiv c-d\pmod n $ and $ k(1-d)\equiv c\pmod n$, then $d(c-d+1)\equiv0\pmod n$, so $h\mid(d-1-c)$. Both are impossible, since $1\leq c-1\leq d-3<h $ and $ 1\leq d-1-c\leq d-3<h$.
Thus no value of $c$ with $2\leq c\leq d-2$ is possible in this case.

The preceding arguments show that every $k\in U_S(\mathbb{Z}_n)$ satisfies $kS=S$ or $kS=-S$. We now determine these units explicitly.
If $kS=S$, then $\{ks,kt\}=\{s,t\}$. Since $s$ and $t$ are units and $s\neq t$, there are two possibilities. 
In the first case, $ks=s$ and $kt=t$. This gives $k=1$. In the second case, $ks=t$ and $kt=s$. The first equality gives $k=ts^{-1}$. With this value of $k$, the second equality is equivalent to $t^2\equiv s^2\pmod n$.
Similarly, if $kS=-S$, then either $k=-1$, or $k=-ts^{-1}$ and $t^2\equiv s^2\pmod n$. Conversely, if $t^2\equiv s^2\pmod n$, then $(ts^{-1})s=t$ and $(ts^{-1})t=s$, so $ts^{-1}S=S$ and $-ts^{-1}S=-S$. Together with the reverse inclusion established above, we obtain
\[
U_S(\mathbb{Z}_n)=
\begin{cases}
\{\pm1,\pm ts^{-1}\}, & t^2\equiv s^2\pmod n,\\[1mm]
\{\pm1\}, & t^2\not\equiv s^2\pmod n.
\end{cases}
\]
This proves Item~\textup{(a)}.

These cases exhaust all possibilities, and therefore prove Theorem~\ref{thm:grand-classification}.
\end{proof}

\section*{Acknowledgements}
Chenhui Lv was supported by the Outstanding Doctoral Students
Overseas Study Program of the University of Science and Technology
of China.
This work was completed during the author's visit to the University
of Melbourne, whose support and hospitality he gratefully
acknowledges. The author is grateful to Binzhou Xia for bringing this
problem to his attention and for providing patient guidance during
the revision of the introduction.

\medskip

\noindent\textbf{Statement on generative AI use}

During the preparation of this work, the author first identified the
need to establish necessary and sufficient conditions for a unit
circulant digraph to admit an endomorphism whose image induces a
directed cycle. The author then used Google Gemini 3.1 Pro to assist
in discovering these conditions, which led to
Theorem~\ref{thm:cycle-image-characterization} and
Corollary~\ref{cor:main-conditions}. Building on this, for reflexive
$3$-valent circulant digraphs that admit such non-trivial
endomorphisms, the author further used Google Gemini 3.1 Pro when
determining the elements of $U_S(\mathbb{Z}_n)$, culminating in
Theorem~\ref{thm:grand-classification}. The author subsequently used
ChatGPT (GPT-5.6 Sol) to assist in substantially simplifying and
streamlining these proofs. The author carefully checked, revised, and
rewrote all resulting arguments, and takes full responsibility for
the final content.

\end{document}